\documentclass[11pt]{article}
\UseRawInputEncoding

\usepackage{mathrsfs}
\usepackage{amssymb}
\usepackage{amsmath}
\usepackage{amsbsy}
\usepackage{epsfig}
\usepackage{enumerate}
\usepackage{bm}
\usepackage{color}
\usepackage[english]{babel}

\usepackage{bbm}
\usepackage{mathrsfs}
\usepackage{dsfont}
\usepackage{graphicx,theorem}

\makeatletter

\newcommand{\Rmnum}[1]{\expandafter\@slowromancap\romannumeral #1@}
\makeatother

\newtheorem{lem}{Lemma}[section]  \newtheorem{thm}{Theorem}[section]
\newtheorem{cor}{Corollary}[section] \newtheorem{defn}{Definition}[section] \newtheorem{rmk}{Remark}[section] 
\numberwithin{equation}{section}
\newenvironment{proof}{{\noindent\it Proof}\quad}{\hfill $\square$\par}

 \newcommand{\me}{\mathrm{e}} 
\newcommand{\dif}{\mathrm{d}} \DeclareMathAlphabet{\mathsfsl}{OT1}{cmss}{m}{sl} \DeclareMathAlphabet{\mathpzc}{OT1}{pzc}{m}{it}

    \newcommand{\ee}{\mathbb{E}}
\newcommand{\hh}{\mathbb{H}}   \newcommand{\nn}{\mathbb{N}} \newcommand{\rr}{\mathbb{R}}
    
\newcommand{\vv}{\mathbb{V}}
 \def\CC{\mathcal C}   \def\FF{\mathcal F}  \def\HH{\mathcal H}

\def\d"{^{\prime\prime}} \def\d'{^{\prime}}

\begin{document}

	\begin{center}{\LARGE\bf Convergence for sums of i. i. d. random variables under sublinear expectations}
	\end{center}
	

	
	\begin{center}
		Mingzhou Xu~\footnote{Email: mingzhouxu@whu.edu.cn} \quad Kun Cheng~\footnote{Email: chengkun0010@126.com}
		\\
		School of Information Engineering, Jingdezhen Ceramic University\\
		Jingdezhen 333403, China
	\end{center}
	
	\renewcommand{\abstractname}{~}
	\begin{abstract}
		{\bf Abstract:}
		In this paper, we obtain the equivalent conditions of complete moment convergence of the maximum for partial weighted sums of independent, identically distributed random variables under sublinear expectations space. The results obtained in the article are the extensions of the equivalent conditions of complete moment convergence of the maximum under classical linear expectation space.
		
		{\bf Keywords:}  Complete moment convergence; Capacity; I. i. d. random variables; Weighted sums; Sublinear expectation
		
		{\bf 2020 Mathematics Subject Classifications:} 60F15, 60F05
		\vspace{-3mm}
	\end{abstract}

	\section{Introduction }
	Peng \cite{Peng2007,Peng2010} first initiated the important concept of  the sublinear expectation space to study the uncertainty of probability and distribution. The seminal works of Peng \cite{Peng2007,Peng2010} attracted people to study inequalities and limit theorems under sublinear expectation space.  Zhang \cite{Zhang2015,Zhang2016a,Zhang2016b} obtained important inequalities including exponential inequalities, Rosenthal's inequalities, and studied Donsker's invariance principle under sublinear expectations. Inspired by the works of Zhang \cite{Zhang2015,Zhang2016a,Zhang2016b,Zhang2016c}, Huang and Wu \cite{Huang2019}, Zhong and Wu \cite{Zhong2017} studied some limits theorems under sublinear expectation space. Recently under sublinear expectations, Wu \cite{Wu2020} proved precise asymptotics for complete integral convergence, Xu and Cheng \cite{Xu2021} establish precise asymptotics in the law of iterated logarithm. Under sublinear expectations for more limit theorems, the interested reader could refer to  Chen \cite{Chen2016}, Xu \cite{Gao2011}, Hu et al. \cite{Hufeng2014}, Hu and Yang \cite{Huzechun2017} and references therein.
	
	Recently Meng et al. \cite{Meng2019} studied convergence for sums of asymptotically almost negatively associated random variables. For references on complete convergence in linear expectation space, the interested reader could refer to Meng et al. \cite{Meng2019}, Shen and Wu \cite{Shen2014} and references therein. The work of Meng et al. \cite{Meng2019} motivate us to wonder whether or not the equivalent conditions of complete moment convergence of the maximum for partial weighted sums of independent, identically distributed random variables under sublinear expectations hold. Here we get that the equivalent conditions of complete moment convergence of the maximum for partial weighted sums of independent, identically distributed random variables  hold under sublinear expectations, which complement the results of Meng et al. \cite{Meng2019} to those under sublinear expectations.
	
	We organize the rest of this paper as follows. In the next section, we recall necessary notions, concepts and relevant properties, and present necessary lemmas under sublinear expectations. In Section 3, we present our main results, Theorems \ref{thm1}, \ref{thm2},  whose proofs are given in Section 4.
	
	\section{Preliminaries}
	We use notations as in the work by Peng \cite{Peng2010}. Suppose that $(\Omega,\FF)$ is a given measurable space. Let $\HH$ be a subset of all random variables on $(\Omega,\FF)$ such that $I_A\in \HH$, where $A\in\FF$, and  $X_1,\cdots,X_n\in \HH$ implies $\varphi(X_1,\cdots,X_n)\in \HH$ for each (local lipschitz) function  $\varphi\in \CC_{l,Lip}(\rr^n)$ satisfying
	$$
	|\varphi(\mathbf{x})-\varphi(\mathbf{y})|\le C(1+|\mathbf{x}|^m+|\mathbf{y}|^m)(|\mathbf{x}-\mathbf{y}|), \forall \mathbf{x},\mathbf{y}\in \rr^n
	$$
	for some $C>0$, $m\in \nn$, which depend on $\varphi$.
	\begin{defn}\label{defn1} A sublinear expectation $\ee$ on $\HH$ is a functional $\ee:\hh\mapsto \bar{\rr}:=[-\infty,\infty]$ satisfying the following properties: for all $X,Y\in \HH$, we have
		\begin{description}
			\item[\rm (a)] If $X\ge Y$, then $\ee[X]\ge \ee[Y]$;
			\item[\rm (b)] $\ee[c]=c$, $\forall c\in\rr$;
			\item[\rm (c)] $\ee[\lambda X]=\lambda\ee[X]$, $\forall \lambda\ge 0$;
			\item[\rm (d)] $\ee[X+Y]\le \ee[X]+\ee[Y]$ whenever $\ee[X]+\ee[Y]$ is not of the form $\infty-\infty$ or $-\infty+\infty$.
		\end{description}
		Here, we refer to {\rm (a)}-{\rm (d)} as  monotonicity, constant preserving, positive homogeneity, subadditivity of $\ee[\cdot]$ respectively.
	\end{defn}
	
	A set function $V:\FF\mapsto[0,1]$ is called a capacity if it obeys
	\begin{description}
		\item[\rm (a)]$V(\emptyset)=0$, $V(\Omega)=1$;
		\item[\rm (b)]$V(A)\le V(B)$, $A\subset B$, $A,B\in \FF$.
	\end{description}
	A capacity $V$ is said to be sub-additive if $V(A+B)\le V(A)+V(B)$, $A,B\in \FF$.

	In this paper, given a sublinear expectation space $(\Omega, \HH, \ee)$, we set a capacity: $\vv(A):=\ee[I_A]$, $\forall A\in \FF$. Clearly $\vv$ is a sub-additive capacity. We set the Choquet expectations $\CC_{\vv}$ by
	$$
	\CC_{\vv}(X):=\int_{0}^{\infty}\vv(X>x)\dif x +\int_{-\infty}^{0}(\vv(X>x)-1)\dif x.
	$$
	
	Given two random vectors $\mathbf{X}=(X_1,\cdots, X_m)$, $X_i\in\HH$ and $\mathbf{Y}=(Y_1,\cdots,Y_n)$, $Y_i\in \HH$  on  $(\Omega, \HH, \ee)$, $\mathbf{Y}$ is said to be independent of $\mathbf{X}$, if for each Borel-measurable function $\psi$ on $\rr^m\times \rr^n$ with $\psi(\mathbf{X},\mathbf{Y}), \psi(\mathbf{x},\mathbf{Y})\in \HH$ for each $x\in\rr^m$, we have $\ee[\psi(\mathbf{X},\mathbf{Y})]=\ee[\ee\psi(\mathbf{x},\mathbf{Y})|_{\mathbf{x}=\mathbf{X}}]$ whenever $\bar{\psi}(\mathbf{x}):=\ee[|\psi(\mathbf{x},\mathbf{Y})|]<\infty$ for each $\mathbf{x}$ and $\ee[|\bar{\psi}(\mathbf{X})|]<\infty$ (cf. Definition 2.5 in Chen \cite{Chen2016} ). $\{X_n\}_{n=1}^{\infty}$ is said to be a sequence of independent random variables, if $X_{n+1}$ is independent of $(X_1,\cdots,X_n)$ for each $n\ge 1$.
	
	Assume that $\mathbf{X}_1$ and $\mathbf{X}_2$ are two $n$-dimensional random vectors defined, respectively, in sublinear expectation spaces $(\Omega_1,\HH_1,\ee_1)$ and $(\Omega_2,\HH_2,\ee_2)$. They are said to be identically distributed if  for every Borel-measurable function $\psi$ such that $\psi(X_1), \psi(X_2)\in \HH$,
	$$
	\ee_1[\psi(\mathbf{X}_1)]=\ee_2[\psi(\mathbf{X}_2)], \mbox{  }
	$$
	whenever the sublinear expectations are finite. $\{X_n\}_{n=1}^{\infty}$ is said to be identically distributed if for each $i\ge 1$, $X_i$ and $X_1$ are identically distributed.
	
	Throughout this paper we assume that $\ee$ is countably sub-additive, i.e., $\ee(X)\le \sum_{n=1}^{\infty}\ee(X_n)$, whenever $X\le \sum_{n=1}^{\infty}X_n$, $X,X_n\in \HH$, and $X\ge 0$, $X_n\ge 0$, $n=1,2,\ldots$.  $C$ represents a positive constant, which may differ in different lines. $I(A)$ or $I_A$ stands for the indicator function of $A$, $a_n\ll b_n$ means that there exists a constant $C>0$ such that $a_n\le C b_n$ for $n$ large enough and let $a_n\approx b_n$ denote that $a_n\ll b_n$ and $b_n\ll a_n$. Denote $\ln\max\{\me, x\}$ by $\log x$.
	
	To establish our results, we need the following lemmas.
	\begin{lem}\label{lem1} Let $Y$ be a random variable under sublinear expectation space $(\Omega,\HH,\ee)$ . Then for any $\alpha>0$, $\gamma>0$, and $\beta>-1$,
		\begin{eqnarray*}
			&({\rm\expandafter{\romannumeral1}})&\int_{1}^{\infty}u^{\beta}\CC_{\vv}(|Y|^{\alpha}I(|Y|>u^{\gamma}))\dif u\le C \CC_{\vv}(|Y|^{(\beta+1)/\gamma+\alpha}),\\
			&({\rm\expandafter{\romannumeral2}})&\int_{1}^{\infty}u^{\beta}\ln(u)\CC_{\vv}(|Y|^{\alpha}I(|Y|>u^{\gamma}))\dif u\le C \CC_{\vv}(|Y|^{(\beta+1)/\gamma+\alpha}\ln(1+|Y|)).
		\end{eqnarray*}
	\end{lem}
	\begin{proof}(\expandafter{\romannumeral1})
		$$
		\begin{aligned}
			&\int_{1}^{\infty}u^{\beta}\CC_{\vv}(|Y|^{\alpha}I(|Y|>u^{\gamma}))\dif u\\
			&=\int_{1}^{\infty}u^{\beta}\left(\int_{0}^{u^{\gamma}}\vv(|Y|>u^{\gamma})\alpha t^{\alpha-1}\dif t+\int_{u^{\gamma}}^{\infty}\vv(|Y|^{\alpha}>t^{\alpha})\alpha t^{\alpha-1}\dif t\right)\dif u\\
			&=\int_{1}^{\infty}u^{\beta+\alpha\gamma}\vv(|Y|>u^{\gamma})\dif u+\int_{1}^{\infty}\alpha t^{\alpha-1}\vv(|Y|>t)\left(\int_{1}^{t^{1/\gamma}}u^{\beta}\dif u\right)\dif t\\
			&\le C \CC_{\vv}(|Y|^{(\beta+1)/\gamma+\alpha})+C\int_{1}^{\infty}t^{\alpha-1+(\beta+1)/\gamma}\vv(|Y|>t)\dif t\\
			&\le C \CC_{\vv}(|Y|^{(\beta+1)/\gamma+\alpha}).
		\end{aligned}
		$$
		(\expandafter{\romannumeral2}) By the proof of Lemma 2.2 in Zhong and Wu \cite{Zhong2017}
		$$
		\begin{aligned}
			&\int_{1}^{\infty}u^{\beta}\ln(u)\CC_{\vv}(|Y|^{\alpha}I(|Y|>u^{\gamma}))\dif u\\
			&=\int_{1}^{\infty}u^{\beta}\ln(u)\left(\int_{0}^{u^{\gamma}}\vv(|Y|>u^{\gamma})\alpha t^{\alpha-1}\dif t+\int_{u^{\gamma}}^{\infty}\vv(|Y|^{\alpha}>t^{\alpha})\alpha t^{\alpha-1}\dif t\right)\dif u\\
			&=\int_{1}^{\infty}u^{\beta+\alpha\gamma}\ln(u)\vv(|Y|>u^{\gamma})\dif u+\int_{1}^{\infty}\alpha t^{\alpha-1}\vv(|Y|>t)\left(\int_{1}^{t^{1/\gamma}}u^{\beta}\ln(u)\dif u\right)\dif t\\
			&\le C \CC_{\vv}(|Y|^{(\beta+1)/\gamma+\alpha}\ln(1+|Y|))+C\int_{1}^{\infty}t^{\alpha-1+(\beta+1)/\gamma}\ln(t)\vv(|Y|>t)\dif t\\
			&\le C \CC_{\vv}(|Y|^{(\beta+1)/\gamma+\alpha}\ln(1+|Y|)).
		\end{aligned}
		$$
	\end{proof}
	Write $S_k=X_1+\cdots+X_k$, $S_0=0$.
	\begin{lem}\label{lem2} (cf. Corollary 2.2 and Theorem 2.3 in Zhang \cite{Zhang2016b}) Suppose that $X_{k+1}$ is independent of $(X_1,\ldots,X_k)$ under sublinear expectation space $(\Omega,\HH,\ee)$ with $\ee(X_i)\le 0$, $k=1,\ldots,n-1$. Then
		\begin{equation}\label{01}
			\ee\left[\left|\max_{k\le n}(S_n-S_k)\right|^M\right]\le 2^{2-M}\sum_{k=1}^{n}\ee[|X_k|^M], \mbox{   for $1\le M\le 2$,}
		\end{equation}
		\begin{equation}\label{02}
			\ee\left[\left|\max_{k\le n}(S_n-S_k)\right|^M\right]\le C_M\left\{\sum_{k=1}^{n}\ee[|X_k|^M]+\left(\sum_{k=1}^{n}\ee[|X_k|^2\right)^{M/2}\right\}, \mbox{   for $ M\ge 2$.}
		\end{equation}
	\end{lem}
\begin{lem}\label{lem2*}(cf. Zhang \cite{Zhang2016b}) Let $\{X_n;n\ge 1\}$ be a sequence of independent random variables under sublinear expectation space $(\Omega,\HH,\ee)$. Then for $M\ge 2$,
	\begin{eqnarray}\label{03}
		\nonumber&&\ee\max_{1\le j\le n}\left|\sum_{i=1}^{j}X_i\right|^M\\
		&&\quad\le C \left(\sum_{i=1}^{n}\ee|X_i|^M+\left(\sum_{i=1}^{n}\ee|X_i|^2\right)^{M/2}+\left(\sum_{i=1}^{n}\left[|\ee(X_i)|+|\ee(-X_i)|\right]\right)^M\right).
	\end{eqnarray}
\end{lem}
	\begin{lem}\label{lem3} Let $Y$ be a random variable under sublinear expectation space $(\Omega,\HH,\ee)$. Then for any $\alpha>0$, $\gamma>0$, and $\beta<-1$,
		\begin{eqnarray}\label{03}
			&&\nonumber({\rm\expandafter{\romannumeral1}})\int_{1}^{\infty}u^{\beta}\ee\left[|Y|^{\alpha}I(|Y|\le u^{\gamma})\right]\dif u\le C \CC_{\vv}\left(|Y|^{(\beta+1)/\gamma+\alpha}\right),\\
			&&({\rm\expandafter{\romannumeral2}})\int_{1}^{\infty}u^{\beta}\ln(u)\ee\left[|Y|^{\alpha}I(|Y|\le u^{\gamma})\right]\dif u\le C \CC_{\vv}\left(|Y|^{(\beta+1)/\gamma+\alpha}\ln(1+|Y|)\right).
		\end{eqnarray}
	\end{lem}
	\begin{proof}
		(\expandafter{\romannumeral1}) By Lemma 4.5 in Zhang \cite{Zhang2016a},
		\begin{eqnarray*}
			&&\int_{1}^{\infty}u^{\beta}\ee\left[|Y|^{\alpha}I(|Y|\le u^{\gamma})\right]\dif u\\
			&&\le \int_{1}^{\infty}u^{\beta}\int_{0}^{u^{\gamma}}\vv\left(|Y|I(|Y|\le u^{\gamma})>t\right)\alpha t^{\alpha-1}\dif t\dif u\\
			&&\le C\int_{0}^{\infty}\vv(|Y|>t) t^{\alpha-1}\int_{1\bigvee t^{1/\gamma}}^{\infty}u^{\beta}\dif u\\
			&&\le C\int_{0}^{\infty}\vv(|Y|>t)t^{\alpha-1+(\beta+1)/\gamma}\dif t\le \CC_{\vv}\left(|Y|^{(\beta+1)/\gamma+\alpha}\right).
		\end{eqnarray*}
		(\expandafter{\romannumeral2}) By Lemma 4.5 in Zhang \cite{Zhang2016a} and the proof of Lemma 2.2 in Zhong and Wu \cite{Zhong2017},
		\begin{eqnarray*}
			&&\int_{1}^{\infty}u^{\beta}\ln(u)\ee\left[|Y|^{\alpha}I(|Y|\le u^{\gamma})\right]\dif u\\
			&&\le \int_{1}^{\infty}u^{\beta}\ln(u)\int_{0}^{u^{\gamma}}\vv\left(|Y|I(|Y|\le u^{\gamma}>t)\right)\alpha t^{\alpha-1}\dif t\dif u\\
			&&\le C\int_{0}^{\infty}\vv(|Y|>t) t^{\alpha-1}\int_{1\bigvee t^{1/\gamma}}^{\infty}u^{\beta}\ln(u)\dif u\\
			&&\le C\int_{0}^{\infty}\vv(|Y|>t)t^{\alpha-1+(\beta+1)/\gamma}\ln(t+1)\dif t\le C \CC_{\vv}\left(|Y|^{(\beta+1)/\gamma+\alpha}\ln(1+|Y|)\right).
		\end{eqnarray*}
	\end{proof}
	\begin{lem}\label{lem4} Let $\{X_n;n\ge 1\}$ be a sequence of independent random variables under sublinear expectation space $(\Omega,\HH,\ee)$. Then the condition that for all $x>0$,
		\begin{eqnarray}\label{04*}
			\lim_{n\rightarrow\infty}\vv\left(\max_{1\le j\le n}|a_{nj}X_j|>x\right)=0
		\end{eqnarray}
		implies that there exists constants $C$ such that for all $x>0$, for $n$ large enough, 
		\begin{eqnarray}\label{04}
			&&\left[1-\vv\left(\max_{1\le j\le n}|a_{nj}X_j|>x\right)\right]^2\sum_{j=1}^{n}\vv(|a_{nj}X_j|>x)\le C\vv\left(\max_{1\le j\le n}|a_{nj}X_j|>x\right).
		\end{eqnarray}
	\end{lem}
	\begin{proof}
		We borrow the idea from Shen and Wu \cite{Shen2014}. Write $A_k=(|a_{nk}X_k|>x)$ and
		$$
		\beta_n=1-\vv(\cup_{k=1}^{n}A_k)=1-\vv\left(\max_{1\le j\le n}|a_{nj}X_j|>x\right)
		$$
		Without restriction of generality, we can assume that $\beta_n>0$. Since $\{I(|a_{nk}X_k|>x)-\ee I(|a_{nk}X_k|>x), k\le 1\}$ is a sequence of independent random variables under sublinear expectations, combining $C_r$'s inequality and Lemma \ref{lem2*} results in
		\begin{eqnarray}\label{05}
			&&\nonumber\ee\left[\sum_{k=1}^{n}(I(A_k)-\ee I(A_k))\right]^2\le C\left\{\sum_{k=1}^{n}\ee\left[(I(A_k)-\ee I(A_k))^2\right]+\left[\sum_{k=1}^{n}\vv(A_k)\right]^2\right\}\\
		\nonumber	&&\le C\sum_{k=1}^{n}\ee\left[(I(A_k)+(\vv(A_k)))^2\right]+C\left[\sum_{k=1}^{n}\vv(A_k)\right]^2\\
			&&\le C\left\{\sum_{k=1}^{n}\vv(A_k)+\left[\sum_{k=1}^{n}\vv(A_k)\right]^2\right\}.
		\end{eqnarray}
		By (\ref{05}), independence of $I(A_k), k=1,\ldots, n$, subadditivity of sublinear expectations, and H\"{o}lder's inequality under sublinear expectations, we conclude that
		\begin{eqnarray*}
			\sum_{k=1}^{n}\vv(A_k)&&=\sum_{k=1}^{n}\ee[I(A_k)]=\ee\left[\sum_{k=1}^{n}I(A_k)\right]\\
			&&=\ee\left[\sum_{k=1}^{n}I(A_k)I\left(\bigcup_{j=1}^{n}A_j\right)\right]\\
			&&\le \ee\left[\sum_{k=1}^{n}(I(A_k)-\ee I(A_k))I\left(\bigcup_{j=1}^{n}A_j\right)\right]+\sum_{k=1}^{n}\vv(A_k)\vv\left(\cup_{j=1}^{n}A_j\right)\\
			&&\le \left[\ee(\sum_{k=1}^{n}(I(A_k)-\ee I(A_k)))^2 \ee\left(I\left(\bigcup_{j=1}^{n}A_j\right)\right)\right]^{1/2}+(1-\beta_n)\sum_{k=1}^{n}\vv(A_k)\\
		&&\le \left\{C(1-\beta_n)\left[\left(\sum_{k=1}^{n}\vv(A_k)\right)+\left(\sum_{k=1}^{n}\vv(A_k)\right)^2\right]\right\}^{1/2}+(1-\beta_n)\sum_{k=1}^{n}\vv(A_k)\\
			&&\le C(1-\beta_n)^{1/2}\sum_{k=1}^{n}\vv(A_k)+\frac12\left[\frac{C(1-\beta_n)}{\beta_n}+\beta_n\sum_{k=1}^{n}\vv(A_k)\right]+(1-\beta_n)\sum_{k=1}^{n}\vv(A_k),
		\end{eqnarray*}
		which combined with (\ref{04*}) results in (\ref{04}) immediately. The proof is finished.
	\end{proof}
	\section{Main results}
	Thoughout the rest of this paper, we assume that $\{X_n,n\ge 1\}$ is a sequence of independent random variables, identically distributed as $X$ under sublinear expectation space  $(\Omega,\HH,\ee)$ with $\ee(X_i)=-\ee(-X_i)= 0$, $i=1,2,\ldots$.
	Our main results are as follows.
	\begin{thm}\label{thm1} Let $\beta>-1$, $r>1$. Let $\{b_{ni}\approx (i/n)^{\beta}(1/n), 1\le i\le n, n\ge 1\}$ satisfying $\sum_{i=1}^{n}b_{ni}=1$ for all $n\ge 1$ hold. Let for $r>1$,
		\begin{eqnarray}\label{06}
			\begin{cases} \CC_{\vv}\left(|X|^{(r-1)/(1+\beta)}\right)<\infty,& \text{
					for $-1<\beta<-1/r$;}\\
				\CC_{\vv}\left(|X|^r\ln(1+|X|)\right)<\infty, & \text{ for $\beta=-1/r$;}\\
				\CC_{\vv}\left(|X|^r\right)<\infty, & \text{ for $\beta>-1/r$,}
			\end{cases}
		\end{eqnarray}
		hold. Then for all $\varepsilon>0$,
		\begin{equation}\label{07}
			\sum_{n=1}^{\infty}n^{r-2}\ee\left\{\left(\max_{1\le j\le n}\left|\sum_{i=1}^{j}b_{ni}X_{i}\right|-\varepsilon\right)^{+}\right\}<\infty,
		\end{equation}
		\begin{equation}\label{08}
			\sum_{n=1}^{\infty}n^{r-2}\vv\left\{\max_{1\le j\le n}\left|\sum_{i=1}^{j}b_{ni}X_{i}\right|>\varepsilon\right\}<\infty.
		\end{equation}
		(\ref{07}) also implies (\ref{06}).
	\end{thm}
	\begin{cor}\label{cor1} Let $\beta>-1$, $r>1$. Assume that $\{b_{ni}\approx[(n-i)/n]^{\beta}(1/n), 0\le i\le n-1, n\ge 1\}$ satisfying $\sum_{i=0}^{n-1}b_{ni}=1$ for all $n\ge 1$. Let (\ref{06}) hold. Then for all $\varepsilon>0$,
		\begin{equation}\label{09}
			\sum_{n=1}^{\infty}n^{r-2}\ee\left\{\left(\max_{0\le j\le n-1}\left|\sum_{i=0}^{j}b_{ni}X_{i}\right|-\varepsilon\right)^{+}\right\}<\infty,
		\end{equation}
		\begin{equation}\label{10}
			\sum_{n=1}^{\infty}n^{r-2}\vv\left\{\max_{0\le j\le n-1}\left|\sum_{i=0}^{j}b_{ni}X_{i}\right|>\varepsilon\right\}<\infty.
		\end{equation}
		(\ref{09}) also  implies (\ref{06}).
	\end{cor}
	For $\alpha>0$, write Ces\`{a}ro summation
	\begin{equation}\label{11}
		A_n^{\alpha}:=\frac{(\alpha+1)(\alpha+2)\cdots(\alpha+n)}{n!}, n=1,2,\ldots, \mbox{   $A_0^{\alpha}=1$.}
	\end{equation}
	\begin{thm}\label{thm2} Let $0<\alpha\le 1$, $r\ge 1$. Let
		\begin{eqnarray}\label{12}
			\begin{cases} \CC_{\vv}\left(|X|^{(r-1)/\alpha}\right)<\infty,& \text{
					for $0<\alpha<1-1/r$;}\\
				\CC_{\vv}\left(|X|^r\ln(1+|X|)\right)<\infty, & \text{ for $\alpha=1-1/r$;}\\
				\CC_{\vv}\left(|X|^r\right)<\infty, & \text{ for $\alpha>1-1/r$}
			\end{cases}
		\end{eqnarray}
		hold. Then for all $\varepsilon>0$,
		\begin{equation}\label{13}
			\sum_{n=1}^{\infty}n^{r-2}\ee\left\{\left(\max_{0\le j\le n}\left|\sum_{i=0}^{j}A_{n-i}^{\alpha-1}X_{i}/A_{n}^{\alpha}\right|-\varepsilon\right)^{+}\right\}<\infty,
		\end{equation}
		\begin{equation}\label{14}
			\sum_{n=1}^{\infty}n^{r-2}\vv\left\{\max_{0\le j\le n}\left|\sum_{i=0}^{j}A_{n-i}^{\alpha-1}X_{i}/A_{n}^{\alpha}\right|>\varepsilon\right\}<\infty.
		\end{equation}
		(\ref{13}) also implies (\ref{12}).
	\end{thm}
	
	In Theorem \ref{thm2}, taking $\alpha=1$, we get the following corollary.
	\begin{cor}\label{cor2} Let $r>1$. Suppose $\CC_{\vv}(|X|^r)<\infty$. Then for all $\varepsilon>0$,
		\begin{equation}\label{15}
			\sum_{n=1}^{\infty}n^{r-2}\ee\left\{\left(\max_{1\le j\le n}\left|\frac{1}{n}\sum_{i=1}^{j}X_{i}\right|-\varepsilon\right)^{+}\right\}<\infty,
		\end{equation}
		\begin{equation}\label{16}
			\sum_{n=1}^{\infty}n^{r-2}\vv\left\{\max_{1\le j\le n}\left|\frac{1}{n}\sum_{i=1}^{j}X_{i}\right|>\varepsilon\right\}<\infty.
		\end{equation}
		(\ref{15}) also concludes $\CC_{\vv}(|X|^r)<\infty$.
	\end{cor}
	
	\begin{rmk}\label{rmk2} Under the same assumptions of Theorem \ref{thm1}, we obtain for all $\varepsilon>0$,
		\begin{eqnarray}\label{17}
			\nonumber\infty&&>\sum_{n=1}^{\infty}n^{r-2}\ee\left\{\left(\max_{1\le j\le n}\left|\sum_{i=1}^{j}b_{ni}X_{i}\right|-\varepsilon/2\right)^{+}\right\}\\
			&&\nonumber\ge\sum_{n=1}^{\infty}2n^{r-2}\vv\left(\max_{1\le j\le n}\left|\sum_{i=1}^{j}b_{ni}X_{i}\right|>\varepsilon\right)/\varepsilon\\
			&&=(2/\varepsilon)\sum_{n=1}^{\infty}n^{r-2}\vv\left(\max_{1\le j\le n}\left|\sum_{i=1}^{j}b_{ni}X_{i}\right|>\varepsilon\right).
		\end{eqnarray}
		By (\ref{17}), we can deduce that (\ref{07}) implies (\ref{08}). Similarly, (\ref{09}) implies (\ref{10}). Hence we complement the results of Meng et al. \cite{Meng2019} to those under sublinear expectations.
	\end{rmk}

	\section{Proofs of Theorems \ref{thm1},\ref{thm2}}
	
	\begin{proof}[Proof of Theorem \ref{thm1}] Here we borrow the idea of proofs of Theorem 16 in Meng et al. \cite{Meng2019}. We first prove that (\ref{06}) implies (\ref{07}) . For all $1\le i\le n$, $n\ge 1$, write
		$$
		Y_{ni}=-\frac{1}{b_{ni}}I(b_{ni}X_i<-1)+X_{i}I(|b_{ni}X_i|\le 1)+\frac{1}{b_{ni}}I(b_{ni}X_i>1).
		$$
		Since $\ee(X_i)=-\ee(-X_i)= 0$, we conclude that
		\begin{eqnarray*}
			\sum_{i=1}^{j}b_{ni}X_i&&\le \sum_{i=1}^{j}b_{ni}(Y_{ni}-\ee Y_{ni})+\sum_{i=1}^{j}\left[|b_{ni}X_i|I(|b_{ni}X_i|>1)+\ee|b_{ni}X_i|I(|b_{ni}X_i|>1)\right] \\
			&&+\sum_{i=1}^{j}\left[I(b_{ni}X_i<-1)+I(b_{ni}X_i>1)\right]+\sum_{i=1}^{j}\vv(|b_{ni}X_i|>1),
		\end{eqnarray*}
		and
		\begin{eqnarray}\label{18}
			\nonumber\max_{1\le j\le n}\left|\sum_{i=1}^{j}b_{ni}X_i\right|&&\le \max_{1\le j\le n}\left|\sum_{i=1}^{j}b_{ni}(Y_{ni}-\ee Y_{ni})\right|+\sum_{i=1}^{n}[I(|b_{ni}X_i|>1)+\vv(|b_{ni}X_i|>1)]\\
			&&+\sum_{i=1}^{n}\left[|b_{ni}X_i|I(|b_{ni}X_i|>1)+\ee|b_{ni}X_i|I(|b_{ni}X_i|>1)\right].
		\end{eqnarray}
		Then
		\begin{eqnarray*}
			\sum_{n=1}^{\infty}n^{r-2}\ee\left(\max_{1\le j\le n}\left|\sum_{i=1}^{j}b_{ni}X_i\right|-\varepsilon\right)^{+}\le && C\sum_{n=1}^{\infty}n^{r-2}\ee\left[\max_{1\le j\le n}\left|\sum_{i=1}^{j}b_{ni}(Y_{ni}-\ee Y_{ni})\right|-\varepsilon\right]^{+} \\
			&&+C\sum_{n=1}^{\infty}n^{r-2}\sum_{i=1}^{n}\vv(|b_{ni}X_i|>1)]\\
			&&+C\sum_{n=1}^{\infty}n^{r-2}\sum_{i=1}^{n}\ee|b_{ni}X_i|I(|b_{ni}X_i|>1)\\
			=:&&\Rmnum{1}+\Rmnum{2}+\Rmnum{3}.
		\end{eqnarray*}
		Thus, to prove (\ref{07}), we need to establish $\Rmnum{1}<\infty$, $\Rmnum{2}<\infty$ and $\Rmnum{3}<\infty$.
		We first prove $\Rmnum{1}<\infty$.
		For fixed $n\ge 1$, since $\{Y_{ni}-\ee Y_{ni},1\le i\le n\}$ is a sequence of independent, identically distributed random variables under sublinear expectation space $(\Omega,\HH,\ee)$, combining Lemma \ref{lem2*}, $C_r$'s inequality, Markov's inequality and Jensen's inequality under sublinear expectations results in
		\begin{eqnarray*}
			\Rmnum{1}&\le&C\sum_{n=1}^{\infty}n^{r-2}\int_{0}^{\infty}\vv\left(\max_{1\le j\le n}\left|\sum_{i=1}^{j}b_{ni}(Y_{ni}-\ee Y_{ni})\right|>t+\varepsilon\right)\dif t\\
			&\le&C\sum_{n=1}^{\infty}n^{r-2} \int_{0}^{\infty}\frac{1}{(t+\varepsilon)^M}\ee\left[\max_{1\le j\le n}\left|\sum_{i=1}^{j}b_{ni}(Y_{ni}-\ee Y_{ni})\right|^M\right]\dif t\\
			&\le& C\sum_{n=1}^{\infty}n^{r-2}\left[\sum_{i=1}^{n}\ee(b_{ni}^M|Y_{ni}-\ee Y_{ni}|^M)+\left(\sum_{i=1}^{n}b_{ni}^2\ee|Y_{ni}-\ee Y_{ni}|^2\right)^{M/2}\right.\\
			&&\left.+\left(\sum_{i=1}^{n}|\ee(b_{ni}Y_{ni})|+|\ee(-b_{ni}Y_{ni})|\right)^M\right]\\
			&\le& C\sum_{n=1}^{\infty}n^{r-2}\sum_{i=1}^{n}\ee(b_{ni}^M|Y_{ni}|^M)+C\sum_{n=1}^{\infty}n^{r-2}\left(\sum_{i=1}^{n}\ee(|b_{ni}Y_{ni}|^2)\right)^{M/2}\\
			&&+C\sum_{n=1}^{\infty}n^{r-2}\left(\sum_{i=1}^{n}|\ee(b_{ni}Y_{ni})|+|\ee(-b_{ni}Y_{ni})|\right)^M=:\Rmnum{1}_1+\Rmnum{1}_2+\Rmnum{1}_3.
		\end{eqnarray*}
		\begin{eqnarray}\label{19}
			\nonumber \Rmnum{1}_1&\le&C\sum_{n=1}^{\infty}n^{r-2}\sum_{i=1}^{n}\vv\left(|b_{ni}X|>1\right)+C\sum_{n=1}^{\infty}n^{r-2}\sum_{i=1}^{n}\ee|b_{ni}X|^MI\left(|b_{ni}X|\le 1\right)=:\Rmnum{1}_{11}+\Rmnum{1}_{12}
		\end{eqnarray}
		By $b_{ni}\approx (i/n)^{\beta}(1/n)$, Lemma \ref{lem1} and (\ref{06}), we see that
		\begin{eqnarray}\label{20}
			\nonumber \Rmnum{1}_{11}&\le&C\sum_{n=1}^{\infty}n^{r-2}\sum_{i=1}^{n}\vv\left(|X_i|>Cn^{1+\beta}i^{-\beta}\right)\\
			\nonumber&\le& C\int_{1}^{\infty}x^{r-2}\int_{1}^{x}\vv\left(|X|>Cx^{1+\beta}y^{-\beta}\right)\dif y\dif x\\
			\nonumber&&\mbox{    ( Setting $s=x^{1+\beta}y^{-\beta}$, $t=y$ )}\\
			\nonumber&\approx&\begin{cases} C\int_{1}^{\infty}s^{\frac{r-1}{1+\beta}-1}\vv\left(|X|>cs\right)\dif s,& \text{
					for $-1<\beta<-1/r$;}\\
				C\int_{1}^{\infty}s^{r-1}\ln(s)\vv\left(|X|>cs\right)\dif s, & \text{ for $\beta=-1/r$;}\\
				C\int_{1}^{\infty}s^{r-1}\vv\left(|X|>cs\right)\dif s, & \text{ for $\beta>-1/r$;}
			\end{cases}\\
			&\approx& \begin{cases} C \CC_{\vv}\left(|X|^{(r+1)/(1+\beta)}\right)<\infty,& \text{
					for $-1<\beta<-1/r$;}\\
				C\CC_{\vv}\left(|X|^r\ln(1+|X|)\right)<\infty, & \text{ for $\beta=-1/r$;}\\
				C\CC_{\vv}\left(|X|^r\right)<\infty, & \text{ for $\beta>-1/r$.}
			\end{cases}
		\end{eqnarray}
		Taking $M$ large enough satisfying $(r-1)/(1+\beta)-1-M<-1$, $r-1-M<-1$, combining Lemma \ref{lem3} and (\ref{06}) results in
		\begin{eqnarray*}
			\Rmnum{1}_{12}&=&C\sum_{n=1}^{\infty}n^{r-2}\sum_{i=1}^{n}n^{-M(1+\beta)}i^{M\beta}\ee\left(|X|^MI(|X|\le Cn^{1+\beta}i^{-\beta})\right)\\
			&\approx&C\int_{1}^{\infty}x^{r-2}\int_{1}^{x}x^{-M(1+\beta)}y^{M\beta}\ee\left(|X|^MI(|X|\le Cx^{1+\beta}y^{-\beta})\right)\dif y\dif x\\
			&&\mbox{     ( Setting  $s=x^{1+\beta}y^{-\beta}$, $t=y$ )}\\
			&\approx&\begin{cases} C\int_{1}^{\infty}s^{\frac{r-1}{1+\beta}-1-M}\CC_{\vv}\left(|X|^MI(|X|>cs)\right)\dif s,& \text{
					for $-1<\beta<-1/r$;}\\
				C\int_{1}^{\infty}s^{r-1-M}\ln(s)\CC_{\vv}\left(|X|^MI(|X|>cs)\right)\dif s, & \text{ for $\beta=-1/r$;}\\
				C\int_{1}^{\infty}s^{r-1-M}\CC_{\vv}\left(|X|^MI(|X|>cs)\right)\dif s, & \text{ for $\beta>-1/r$;}
			\end{cases}\\
			&\le&\begin{cases} C \CC_{\vv}\left(|X|^{(r+1)/(1+\beta)}\right)<\infty,& \text{
					for $-1<\beta<-1/r$;}\\
				C\CC_{\vv}\left(|X|^r\ln(1+|X|)\right)<\infty, & \text{ for $\beta=-1/r$;}\\
				C\CC_{\vv}\left(|X|^r\right)<\infty, & \text{ for $\beta>-1/r$.}
			\end{cases}
		\end{eqnarray*}
		We next prove $\Rmnum{1}_2<\infty$. By $C_r$'s inequality, we see that
		\begin{eqnarray}\label{21}
			\nonumber \Rmnum{1}_2&=&C\sum_{n=1}^{\infty}n^{r-2}\left[\sum_{i=1}^{n}\vv\left(|b_{ni}X_i|>1\right)+\sum_{i=1}^{n}\ee|b_{ni}X_i|^2I\left(|b_{ni}X_i|\le 1\right)\right]^{M/2}\\
			\nonumber &\le&C\sum_{n=1}^{\infty}n^{r-2}\left[\sum_{i=1}^{n}\vv\left(|b_{ni}X|>1\right)+\sum_{i=1}^{n}\ee|b_{ni}X|^2I\left(|b_{ni}X|\le 1\right)\right]^{M/2}\\
			\nonumber &\le&C\sum_{n=1}^{\infty}n^{r-2}\left[\sum_{i=1}^{n}\vv\left(|b_{ni}X|>1\right)\right]^{M/2}+C\sum_{n=1}^{\infty}n^{r-2}\left[\sum_{i=1}^{n}\ee|b_{ni}X|^2I\left(|b_{ni}X|\le 1\right)\right]^{M/2}\\
			&=&:\Rmnum{1}_{21}+\Rmnum{1}_{22}.
		\end{eqnarray}
		Taking $M$ large enough satisfying $r-2-(r-1)M/2<-1$, combining Markov's inequality under sublinear expectations and (\ref{06}) results in
		\begin{eqnarray*}
			\Rmnum{1}_{21}&\approx&C\sum_{n=1}^{\infty}n^{r-2}\left[\sum_{i=1}^{n}\vv\left(|X|>Cn^{1+\beta}i^{-\beta}\right)\right]^{M/2}\\
			&\le&\begin{cases} C\sum_{n=1}^{\infty}n^{r-2}\left[\sum_{i=1}^{n}\left(\frac{i^{\beta}}{n^{1+\beta}}\right)^{(r-1)/(1+\beta)}\right]^{M/2},& \text{
					for $-1<\beta<-1/r$;}\\
				C\sum_{n=1}^{\infty}n^{r-2}\left[\sum_{i=1}^{n}\left(\frac{i^{\beta}}{n^{1+\beta}}\right)^{r}\ln\left(1+\frac{i^{\beta}}{n^{1+\beta}}\right)\right]^{M/2}, & \text{ for $\beta=-1/r$;}\\
				C\sum_{n=1}^{\infty}n^{r-2}\left[\sum_{i=1}^{n}\left(\frac{i^{\beta}}{n^{1+\beta}}\right)^{r}\right]^{M/2}, & \text{ for $\beta>-1/r$;}
			\end{cases}\\
			&\le&\begin{cases} C\sum_{n=1}^{\infty}n^{r-2-(r-1)M/2}<\infty,& \text{
					for $-1<\beta<-1/r$;}\\
				C\sum_{n=1}^{\infty}n^{r-2-(r-1)M/2}(\ln n)^{M/2}<\infty, & \text{ for $\beta=-1/r$;}\\
				C\sum_{n=1}^{\infty}n^{r-2-(r-1)M/2}<\infty, & \text{ for $\beta>-1/r$.}
			\end{cases}
		\end{eqnarray*}

		We next establish $\Rmnum{1}_{22}<\infty$ in the following two cases.\\
		(\expandafter{\romannumeral1}) If $1<r<2$, taking $M$ large satisfying $r-2-Mr(1+\beta)/2<-1$, $r-2-M(r-1)/2<-1$, then by $\ee(|X|^r)\le \CC_{\vv}(|X|^r)<\infty$, we see that
		\begin{eqnarray*}
			\Rmnum{1}_{22}&\le& C\sum_{n=1}^{\infty}n^{r-2}\left(\sum_{i=1}^{n}b_{ni}^{r}\right)^{M/2}\\
			&\approx& C\sum_{n=1}^{\infty}n^{r-2}\left(\sum_{i=1}^{n}n^{-r(1+\beta)}i^{r\beta}\right)^{M/2}\\
			&\approx&\begin{cases} C\sum_{n=1}^{\infty}n^{r-2-Mr(1+\beta)/2}<\infty,& \text{
					for $-1<\beta<-1/r$;}\\
				C\sum_{n=1}^{\infty}n^{r-2-(r-1)M/2}(\ln n)^{M/2}<\infty, & \text{ for $\beta=-1/r$;}\\
				C\sum_{n=1}^{\infty}n^{r-2-(r-1)M/2}<\infty, & \text{ for $\beta>-1/r$.}
			\end{cases}
		\end{eqnarray*}
		(\expandafter{\romannumeral2}) If $r\ge 2$, then $(\ref{06})$ yields $\ee(X^2)<\infty$. Taking $M$ large satisfying $r-2-M(1+\beta)<-1$, $r-1-p/2<-1$, we deduce that
		\begin{eqnarray*}
			\Rmnum{1}_{22}&\le& C\sum_{n=1}^{\infty}n^{r-2}\left(\sum_{i=1}^{n}b_{ni}^{2}\right)^{M/2}\\
			&\approx& C\sum_{n=1}^{\infty}n^{r-2}\left(\sum_{i=1}^{n}n^{-2(1+\beta)}i^{2\beta}\right)^{M/2}\\
			&\approx&\begin{cases} C\sum_{n=1}^{\infty}n^{r-2-M(1+\beta)}<\infty,& \text{
					for $-1<\beta<-1/r$;}\\
				C\sum_{n=1}^{\infty}n^{r-2-M/2}(\ln n)^{M/2}<\infty, & \text{ for $\beta=-1/r$;}\\
				C\sum_{n=1}^{\infty}n^{r-2-M/2}<\infty, & \text{ for $\beta>-1/r$.}
			\end{cases}
		\end{eqnarray*}
	We next prove $\Rmnum{1}_3<\infty$. Note that (\ref{06}) yields $\ee(|X|^r)<\infty$. By $\ee(X_i)=\ee(-X_i)=0$,
	Markov's inequality under sub-linear expectations, choosing  $M$ large enough such that $r-2-[r(\beta+1)]M<-1$, we see that
	\begin{eqnarray}\label{21*}
		\nonumber&&\Rmnum{1}_3\le C\sum_{n=1}^{\infty}n^{r-2}\left(\sum_{i=1}^{n}|\ee(b_{ni}Y_{ni}-b_{ni}X_i)|+|\ee(-b_{ni}Y_{ni}+b_{ni}X_i)|\right)^M\\
		\nonumber&&\le C\sum_{n=1}^{\infty}n^{r-2}\left(\sum_{i=1}^{n}\ee|b_{ni}Y_{ni}-b_{ni}X_i|\right)^M\le C\sum_{n=1}^{\infty}n^{r-2}\left(\sum_{i=1}^{n}\ee|b_{ni}X_{i}|^{r}\right)^M\\
		\nonumber&&\quad \le C\sum_{n=1}^{\infty}n^{r-2}\left(\sum_{i=1}^{n}i^{\beta r}n^{-r(\beta+1)}\right)^M\\
		&&\quad \le \begin{cases}
			C\sum_{n=1}^{\infty}n^{r-2-[r(\beta+1)] M}, &\text{  for $-1<\beta<-1/r$,}\\
			C\sum_{n=1}^{\infty}n^{r-2-[r-1] M}(\log n)^M, &\text{  for $\beta=-1/r$,}\\
			C\sum_{n=1}^{\infty}n^{r-2-[r-1] M}, &\text{  for $\beta>-1/r$}
		\end{cases}\\
		\nonumber&&\quad<\infty.
	\end{eqnarray} 
Hence Therefore, by (\ref{19})-(\ref{21*}), we deduce that $I<\infty$.
	
		By the proof of $\Rmnum{1}_{11}<\infty$, we see that $\Rmnum{2}<\infty$.
		We finally establish $\Rmnum{3}<\infty$.  Since $b_{ni}\approx (i/n)^{\beta}(1/n)$, combining Lemma \ref{lem1}, and Lemma 4.5 in Zhang \cite{Zhang2016a} results in
		\begin{eqnarray*}
			\nonumber \Rmnum{3}&\le&C\sum_{n=1}^{\infty}n^{r-2}\sum_{i=1}^{n}n^{-(1+\beta)}i^{\beta}\ee\left(|X|I(|X|>Cn^{1+\beta}i^{-\beta})\right)\\
			\nonumber&\le& C\int_{1}^{\infty}x^{r-2}\int_{1}^{x}x^{-(1+\beta)}y^{\beta}\CC_{\vv}\left(|X|I(|X|>Cx^{1+\beta}y^{-\beta})\right)\dif y\dif x\\
			\nonumber&&\mbox{    ( Setting $s=x^{1+\beta}y^{-\beta}$, $t=y$ )}\\
			\nonumber&\approx&\begin{cases} C\int_{1}^{\infty}s^{\frac{r-1}{1+\beta}-2}\CC_{\vv}\left(|X|I(|X|>cs)\right)\dif s,& \text{
					for $-1<\beta<-1/r$;}\\
				C\int_{1}^{\infty}s^{r-2}\ln(s)\CC_{\vv}\left(|X|I(|X|>cs)\right)\dif s, & \text{ for $\beta=-1/r$;}\\
				C\int_{1}^{\infty}s^{r-2}\CC_{\vv}\left(|X|I(|X|>cs)\right)\dif s, & \text{ for $\beta>-1/r$;}
			\end{cases}\\
			&\le& \begin{cases} C \CC_{\vv}\left(|X|^{(r+1)/(1+\beta)}\right)<\infty,& \text{
					for $-1<\beta<-1/r$;}\\
				C\CC_{\vv}\left(|X|^r\ln(1+|X|)\right)<\infty, & \text{ for $\beta=-1/r$;}\\
				C\CC_{\vv}\left(|X|^r\right)<\infty, & \text{ for $\beta>-1/r$.}
			\end{cases}
		\end{eqnarray*}
		
		We now prove that (\ref{07}) concludes (\ref{06}). By Remark \ref{rmk2}, we see that (\ref{07}) gives
		\begin{eqnarray}\label{22}
			\sum_{n=1}^{\infty}n^{r-2}\vv\left(\max_{1\le j\le n}\left|\sum_{i=1}^{j}b_{ni}X_{i}\right|>\varepsilon\right)<\infty.
		\end{eqnarray}
		Observing that
		\begin{eqnarray}\label{23}
			\max_{1\le j\le n}|b_{nj}X_{j}|\le\max_{1\le j\le n}\left|\sum_{i=1}^{j}b_{ni}X_{i}\right|,
		\end{eqnarray}
		by (\ref{22}) we obtain
		\begin{eqnarray}\label{24}
			\sum_{n=1}^{\infty}n^{r-2}\vv\left(\max_{1\le j\le n}|b_{nj}X_{j}|>\varepsilon\right)<\infty,
		\end{eqnarray}
		\begin{eqnarray}\label{25}
			\vv\left(\max_{1\le j\le n}|b_{nj}X_{j}|>\varepsilon\right)\rightarrow 0  \mbox{    as $n\rightarrow\infty$.}
		\end{eqnarray}
		By Lemma \ref{lem4} and (\ref{25}), we see that
		\begin{eqnarray}\label{26}
			\sum_{i=1}^{n}\vv\left(|b_{ni}X_{i}|>\varepsilon\right)\le C\vv\left(\max_{1\le j\le n}|b_{nj}X_{j}|>\varepsilon\right).
		\end{eqnarray}
		Hence, combining (\ref{26}) and (\ref{24}) results in
		\begin{eqnarray}\label{27}
			\sum_{n=1}^{\infty}n^{r-2}\sum_{i=1}^{n}\vv\left(|b_{ni}X_{i}|>\varepsilon\right)<\infty.
		\end{eqnarray}
		As in the proof of $\Rmnum{1}_{11}<\infty$, we obtain
		\begin{eqnarray*}
			\infty&>&\sum_{n=1}^{\infty}n^{r-2}\sum_{i=1}^{n}\vv\left(|b_{ni}X_{i}|>\varepsilon\right)\\
			&\approx&\begin{cases} C \CC_{\vv}\left(|X|^{(r+1)/(1+\beta)}\right)<\infty,& \text{
					for $-1<\beta<-1/r$;}\\
				C\CC_{\vv}\left(|X|^r\ln(1+|X|)\right)<\infty, & \text{ for $\beta=-1/r$;}\\
				C\CC_{\vv}\left(|X|^r\right)<\infty, & \text{ for $\beta>-1/r$.}
			\end{cases}
		\end{eqnarray*}
		Consequently, this finishes the proof of Theorem \ref{thm1}.
	\end{proof}
	\begin{proof}[Proof of Corollary \ref{cor1}] As in the proof of  Theorem 16 in Meng et al. \cite{Meng2019}, the proof is similar to that of Theorem \ref{thm1}, so it is omitted.
	\end{proof}
	\begin{proof}[Proof of Theorem \ref{thm2}] Taking $b_{ni}=A_{n-i}^{\alpha-1}/A_{n}^{\alpha}$, $0\le i\le n$, $n\ge1$, by the similar proof of Theorem 18 in Meng et al. \cite{Meng2019}, we deduce that the assumptions of Corollary \ref{cor1} hold. Hence (\ref{13}) follows from (\ref{09}).
		The proof of Theorem \ref{thm2} is finished.
	\end{proof}

	{\bf Acknowledgements}
	
	Not applicable.
	
	{\bf Funding}
	
	This research was supported by Doctoral Scientific Research Starting Foundation of Jingdezhen Ceramic University ( Nos.102/01003002031 ), Scientific Program of Department of Education of Jiangxi Province of China (Nos. GJJ190732, GJJ180737), Natural Science Foundation Program of Jiangxi Province 20202BABL211005, and National Natural Science Foundation of China (Nos. 61662037).
	
	{\bf Availability of data and materials}
	
	No data were used to support this study.
	
	{\bf Competing interests}
	
	The authors declare that they have no competing interests.
	
	{\bf Authors¡¯ contributions}
	
	All authors contributed equally and read and approved the final manuscript.

\end{document}